\documentclass{article}
\usepackage{amsmath}
\usepackage{amsfonts}
\usepackage{amssymb}
\usepackage{graphicx}
\usepackage{color}%
\newcommand{\res}{\mathop{\hbox{\vrule height 7pt width .5pt depth 0pt
\vrule height .5pt width 6pt depth 0pt}}\nolimits}
\newtheorem{theorem}{Theorem}

\newtheorem{definition}[theorem]{Definition}
\newtheorem{lemma}[theorem]{Lemma}
\newtheorem{proposition}[theorem]{Proposition}
\newtheorem{remark}[theorem]{Remark}
\newenvironment{proof}[1][Proof]{\noindent\textbf{#1.} }{\ \rule{0.5em}{0.5em}}
\begin{document}

\title{Two characterization of $BV$ functions on Carnot groups via the heat semigroup}
\author{Marco Bramanti\thanks{Dipartimento di Matematica, Politecnico di Milano, Via
Bonardi 9, 20133 Milano, Italy. e-mail: marco.bramanti@polimi.it}
\and Michele Miranda Jr\thanks{Dipartimento di Matematica, Universit\`a di Ferrara,
via Machiavelli, 35, 44100, Ferrara, Italy. e-mail: michele.miranda@unife.it}
\and Diego Pallara\thanks{Dipartimento di Matematica ``Ennio De Giorgi'',
Universit\`a del Salento, C.P.193, 73100, Lecce, Italy. e-mail:
diego.pallara@unisalento.it}}
\maketitle

\begin{abstract}
In this paper we provide two different characterizations of
sets with finite perimeter and functions of bounded variation
in Carnot groups, analogous to those which hold in Euclidean spaces, in terms
of the short-time behaviour of the heat semigroup. The second one holds under
the hypothesis that the reduced boundary of a set of finite perimeter is
rectifiable, a result that presently is known in Step 2 Carnot groups.\\
\textbf{MSC} (2000): 22E30, 28A75, 35K05, 49Q15.\\
\textbf{Keywords}: Functions of bounded variation, perimeters, Carnot groups, Heat semigroup.
\end{abstract}

\maketitle

\section{Introduction}

There are various definitions of \emph{variation} of a function, and the
related classes $BV$ of \emph{functions of bounded variation}, that make sense
in different contexts and are known to be equivalent in wide generality. In
the Euclidean framework, i.e., ${{\mathbb{R}}^{n}}$ endowed with the Euclidean
metric and the Lebesgue measure, the variation of $f\in L^{1}({{\mathbb{R}}^{n}})$ 
can be defined as
\begin{equation} \label{variation1}
\left\vert Df\right\vert ({{\mathbb{R}}^{n}})=\sup\left\{  
\int_{{{\mathbb{R}}^{n}}}f\mathrm{div}\,g\ dx:\ 
g\in C_{0}^{1}({{\mathbb{R}}^{n}},{{\mathbb{R}}^{n}}),\ 
\Vert g\Vert_{L^{\infty}({{\mathbb{R}}^{n}})}\leq1\right\}  ,
\end{equation}
and $|Df|({{\mathbb{R}}^{n}})<+\infty$ is equivalent to saying that
$f\in BV({\mathbb R}^n)$, i.e., the distributional gradient of $f$ is a
(${{\mathbb{R}}^{n}}$-valued) finite Radon measure. This approach can be
generalized to ambients where a measure and a coherent differential structure,
which allows to define a divergence operator, are defined. Alternatively, the
variation of $f$ can be defined through a relaxation procedure,
\begin{equation} \label{variation2}
\left\vert Df\right\vert ({{\mathbb{R}}^{n}})=\inf\left\{  
\liminf_{k\rightarrow\infty}\int_{{{\mathbb{R}}^{n}}}|\nabla f_{k}|dx:\ 
f_{k}\in\mathrm{Lip}({{\mathbb{R}}^{n}}),\ f_{k}\rightarrow f\ \mathrm{in\ }
L^{1}({{\mathbb{R}}^{n}})\right\}  ,
\end{equation}
and the space $BV$ can be defined, accordingly, as the finiteness domain of
the relaxed functional in $L^{1}({{\mathbb{R}}^{n}})$. To generalize
\eqref{variation2}, no differential structure is needed, apart from Lipschitz
functions and a suitable substitute of the modulus of the gradient.
Definitions \eqref{variation1} and \eqref{variation2} have been extended to
several contexts, such as manifolds and metric measure spaces (see e.g.
\cite{miranda}, \cite{metrsbv}, \cite{KiKoShTu10}), and in particular Carnot-Carath\'{e}odory
spaces and Carnot groups, see \cite{CapDanGar94The}, \cite{BirMos95Sob},
\cite{fraserser}. However, we point out that the original definition of
variation of a function, and in particular of set of finite perimeter, has
been given by E. De Giorgi in \cite{DG1} by a regularization procedure based
on the heat kernel. He proved that
\begin{equation}   \label{variation3}
|Df|({{\mathbb{R}}^{n}})=\lim_{t\rightarrow0}
\int_{{{\mathbb{R}}^{n}}}|\nabla T_{t}f|dx,
\end{equation}
(where $\nabla$ denotes the gradient with respect to the space variables
$x\in{{\mathbb{R}}^{n}}$ and $T_{t}f(x)=\int_{{{\mathbb{R}}^{n}}}h(t,x-y)f(y)dy$ 
is the heat semigroup given by the Gauss-Weierstrass kernel
$h$) and that if the above quantity is finite then $f\in BV({\mathbb R}^n)$.
Equality \eqref{variation3} has been recently proved for Riemannian manifolds
$M$ in \cite{CarMau07}, with the only restriction that the Ricci curvature of
$M$ is bounded from below, thus generalizing the result in
\cite{MirPalParPre07Rie}, where further bounds on the geometry of $M$ were assumed.

In this paper the framework is that of Carnot groups, where it is known that
formulae \eqref{variation1} and \eqref{variation2} (where of course the
intrinsic differential structure is used) give equivalent definitions of
$|Df|$, see \cite{fraserser}; we show here that \eqref{variation3} holds in a
weaker sense, all the quantities involved being again the intrinsic ones.
Namely, the two sides of \eqref{variation3} are equivalent, although
we don't know whether they are equal (see Theorem~\ref{Main Theorem 1} 
for the precise statement).

We think that it is interesting to comment on the proof of \eqref{variation3}
in the various situations. Thinking of the left hand side as defined by
\eqref{variation1}, the $\leq$ inequality follows easily by lower
semicontinuity of the total variation and the $L^{1}$ convergence of $T_{t}f$
to $f$, so that the more difficult part is the $\geq$ inequality. In
${{\mathbb{R}}^{n}}$ it follows from the (trivial) commutation relation
$D_{i}T_{t}f=T_{t}D_{i}f$ between the heat semigroup and the partial
derivatives, whereas in the Riemannian case, as treated in \cite{CarMau07}, it
follows from the estimate $\Vert dT_{t}f\Vert_{L^{1}(M)}\leq e^{K^{2}t}|df|(M)$, 
where $d$ denotes the exterior differential on $M$, which replaces
the gradient in \eqref{variation3}, and $-K^{2}$ is a lower bound for the
Ricci curvature of $M$. In the subriemannian case we are dealing with, the
Euclidean commutation does not hold, and we are not able to prove that
equality \eqref{variation3} still holds, but only that the semigroup approach
identifies the $BV$ class. The proof relies on algebraic properties and
Gaussian estimates on the heat kernel and its derivatives which allow to
estimate the commutator $[D,T_{t}]$. Still different arguments are used in the
infinite dimensional case of Wiener spaces.

Recently, inspired by \cite{Led}, another connection between the heat
semigroup and $BV$ functions in ${{\mathbb{R}}^{n}}$ has been pointed out in
\cite{MirPalParPre07Sho}, where the following equality is proved:
\begin{equation}   \label{diffusion}
|Df|({{\mathbb{R}}^{n}})=\lim_{t\rightarrow0}\frac{\sqrt{\pi}}{2\sqrt{t}}
\int_{{{\mathbb{R}}^{n}}}\int_{{{\mathbb{R}}^{n}}}h(t,x-y)|f(x)-f(y)|\,dydx,
\end{equation}
which means that $f\in BV({{\mathbb{R}}^{n}})$ if and only if the right hand
side is finite, and equality holds. Equality \eqref{diffusion} is first proved
for characteristic functions and then extended to the general case using the
coarea formula. The proof for characteristic functions of sets of finite
perimeter, in turn, is based on a blow-up argument on the points of the
reduced boundary and uses the rectifiability of the reduced boundary. In this
paper we extend \eqref{diffusion} to Carnot groups (see Theorem~\ref{Main Theorem 2}) 
but, accordingly to the preceding comments, our proof
depends upon the rectifiability of the reduced boundary, which is presently
known to be true only in groups of Step 2, see \cite{fraserser3}. Therefore,
our proof would immediately generalize to more general groups, if the
rectifiability of the reduced boundary were proved. Notice also that in general 
the heat kernel doesn't have the symmetries that the Gaussian kernel has, such as rotation
invariance. As a consequence, equation \eqref{diffusion} assumes a different 
form in general Carnot groups of step two (see \eqref{diffusionbv} in Theorem 
\ref{Main Theorem 2} below); however, in the special but important case of groups 
of Heisenberg type, \eqref{diffusionbv} simplifies to a form
very similar to \eqref{diffusion} (see Remark \ref{Remark H-type}).

We point out that in the Euclidean case both \eqref{variation3} and
\eqref{diffusion} hold not only in the whole of ${{\mathbb{R}}^{n}}$, but also
in a localized form and with the heat semigroup replaced by the semigroup
generated by a general uniformly elliptic operator, under suitable boundary
conditions, see \cite{Angiuli_Miranda_Pallara_Paronetto:2007}. Finally, let us mention that the
short-time behavior of the heat semigroup in ${\mathbb{R}}^{n}$ has been shown
to be useful also to describe further geometric properties of boundaries, see
\cite{AMM}.

\bigskip

{\bf Aknowledgments} Miranda was partially supported by GNAMPA pro\-ject
\textquotedblleft Metodi geometrici per analisi in spazi non Euclidei; spazi
metrici doubling, gruppi di Carnot e spazi di Wiener\textquotedblright. The
authors would like to thanks L.~Capogna, N.~Garofalo, G.~Leo\-nar\-di, V.~Magnani, R.~Monti,
R.~Serapioni and D.~Vittone for many helpful discussions, as well as the anonymous
referees for their useful suggestions.

\section{Preliminaries and Main Results}

We collect in the first two subsections the notions on Carnot groups and $BV$
functions needed in the paper. After that, we devote the last subsection to
the statement of the main results and some comments.

\subsection{Carnot groups and heat kernels\label{sec carnot}}

\textbf{Basic definitions. }Here we recall some basic known facts about Carnot
groups, whose precise definition we shall give after introducing some
notation. We refer to \cite[\S \ 1.4]{BLUbook} for the proofs and further details.

Let $\mathbb{R}^{n}$ be equipped with a Lie group structure by the composition
law $\circ$ (which we call \emph{translation}) such that $0$ is the identity
and $x^{-1}=-x$ (i.e., the group inverse is the Euclidean opposite); let
$\mathbb{R}^{n}$ be also equipped with a family 
$\left\{  D\left(\lambda\right)  \right\}_{\lambda>0}$ of group automorphisms of $\left(
\mathbb{R}^{n},\circ\right)  $ (called \emph{dilations}) of the following form 
\[
D\left(  \lambda\right)  :\left(  x_{1},\ldots,x_{n}\text{ }\right)
\mapsto\text{ }\left(  \lambda^{\omega_{1}}x_{1},\ldots,\lambda^{\omega_{n}}x_{n}\right)
\]
where $\omega_{1}\leq\omega_{2}\leq\ldots\leq\omega_{n}$ are positive
integers, with $\omega_{1}=\omega_{2}=\ldots=\omega_{q}=1,\omega_{i}>1$ for
$i>q$ for some $q<n.$

If $L_{x}$ is the left translation operator acting on functions,
$(L_{x}f)(y)=f(x\circ y)$, we say that a \emph{differential operator }
$P$\emph{ is left invariant }if \ $P(L_{x}f)=L_{x}(Pf)$ for every smooth
function $f$. Also, we say that \emph{a differential operator }$P$\emph{ is
homogeneous of degree} $\delta>0$ if
\[
P\,(f(D(\lambda)))(x) = \lambda^{\delta}\,(Pf)(D(\lambda)x)
\]
for every test function $f$, $\lambda>0$, $x\in{{\mathbb{R}}^{n}}$, and
\emph{a function }$f$\emph{ is homogeneous of degree} $\delta\in\mathbb{R}$
if
\[
f\,(D(\lambda)x)=\lambda^{\delta}\,f\,(x)\text{ \ for every }\lambda>0,\ 
x\in{{\mathbb{R}}^{n}}.
\]
Now, for $i=1,2,...,q,$ let
\begin{equation}
X_{i}=\sum_{j=1}^{n}q_{i}^{j}\left(  x\right)  \partial_{x_{j}} \label{X}
\end{equation}
be the unique left invariant vector field (with respect to $\circ$) which
agrees with $\partial_{x_{i}}$ at the origin, and assume that the Lie algebra
generated by $X_{1},X_{2},...,X_{q}$ coincides with the Lie algebra of
$\mathbb{G}.$ Then we say that $\mathbb{G}=(\mathbb{R}^{n},\circ,\{D(\lambda)\}_{\lambda>0})$ 
is a \emph{(homogeneous) Carnot group }or a \emph{homogeneous stratified group.}

More explicitly, setting $V_{1}=\mathrm{span}\{X_{1},X_{2},...,X_{q}\}$,
$V_{i+1}=[V_{1},V_{i}]$ (the space generated by the commutators $[X,Y]$, with
$X\in V_{1},\ Y\in V_{i}$), there is $k\in\mathbb{N}$ such that 
${{\mathbb{R}}^{n}}=V_{1}\oplus\cdots\oplus V_{k}$, with $V_{k}\neq\{0\}$ and $V_{i}=\{0\}$
for $i>k$. The integer $k$ is called the \emph{step} of $\mathbb{G}$. 
Since $V_1$ at 0 can be identified with ${\mathbb R}^q$, by the left invariance 
of the vector fields $X_i$ we can identify $V_1$ with ${\mathbb R}^q$ for every
$x\in {\mathbb R}^n$. Note that all vector fields in $V_{j}$ are $j$-homogeneous and call
\[
Q=\sum_{j=1}^{k}j\mathrm{dim}\,V_{j}=\sum_{j=1}^{n}\omega_{j}
\]
the \emph{homogeneous dimension }of ${\mathbb{G}}$.
The operator
\[
L=\sum_{j=1}^{q}X_{j}^{2}
\]
is called \emph{sublaplacian}; it is left invariant, homogeneous of degree
two and, by H\"{o}rmander's Theorem, hypoelliptic.

\textbf{Structure of left- and right-invariant vector fields. }It is sometimes
useful to consider also the \emph{right-invariant vector fields} $X_{j}^{R}$
($j=1,\ldots,n$), which agree with $\partial_{x_{j}}$ (and therefore with
$X_{j}$) at $0$; also these $X_{i}^{R}$ are homogeneous of degree one for
$i\leq q$. (For the following properties of invariant vector fields see \cite[pp.606-621]{St} or \cite{BB}). 
As to the \emph{structure of the left (or right) invariant vector fields}, it can be proved that the systems 
$\left\{ X_{i}\right\}  $ and $\left\{  X_{i}^{R}\right\}  $ have the following
\textquotedblleft triangular form\textquotedblright\ with respect to Cartesian
derivatives:
\begin{align}\label{X triang} 
X_{i}  &  =\partial_{x_{i}}+\sum_{k=i+1}^{n}q_{i}^{k}(x)\, \partial_{x_{k}}
\\  \label{X R triang}
X_{i}^{R}  &  =\partial_{x_{i}}+
\sum_{k=i+1}^{n}\overline{q}_{i}^{k}(x)\,\partial_{x_{k}}%
\end{align}
where $q_{i}^{k},\overline{q}_{i}^{k}$ are polynomials, homogeneous of degree
$\omega_{k}-\omega_{i}$ (the $\omega_{i}$'s are the dilation exponents). When
the triangular form of the $X_{i}$'s is not important, we will keep writing
(\ref{X}), more compactly.

The identities (\ref{X triang})-(\ref{X R triang}) imply that any Cartesian
derivative $\partial_{x_{k}}$ can be written as a linear combination of the
$X_{i}$'s (and, analogously, of the $X_{j}^{R}$'s). In particular, any
homogeneous differential operator can be rewritten as a linear combination of
left invariant (or, similarly, right invariant) homogeneous vector fields,
with polynomial coefficients. The above structure of the $X_{i}$'s also
implies that the $L^2$ transpose $X_{i}^{\ast}$ of $X_{i}$ is just $-X_{i}$
(as in a standard integration by parts). From the above equations we also find
that
\[
X_{i}=\sum_{k=i}^{n}\text{ }c_{i}^{k}(x)\,X_{k}^{R}
\]
where $c_{i}^{k}(x)$ are polynomials, homogeneous of degree $\omega_{k}
-\omega_{i}$. In particular, since $\omega_{k}-\omega_{i}<\omega_{k}$,
$c_{i}^{k}(x)$ does not depend on $x_{h}$ for $h\geq k$ and therefore commutes
with $X_{k}^{R}$, that is
\[
X_{i}f=\sum_{k=i}^{n}X_{k}^{R}\,\left(  c_{i}^{k}
(x)\,f\right)  \text{ \ \ \ \ (}i=1,\ldots,n\text{)}
\]
for every test function $f$. The above identity can be sharpened as follows;
taking its value at the origin we get:
\[
\partial_{x_{i}}=\sum_{k=i}^{q}c_{i}^{k}\partial_{x_{k}}\,
\]
since for $k>q$ the $c_{i}^{k}(x)$ are homogeneous polynomials of positive
degree, and therefore vanish at the origin. Hence $c_{i}^{k}=\delta_{ik}$ for
$i,k=1,2,...,q,$ that is:
\begin{equation}    \label{Xi_XiR}
X_{i}f=X_{i}^{R}f+ \sum_{k=q+1}^{n}X_{k}^{R}\,
\left(  c_{i}^{k}(x)\,f\right)  \text{ \ \ \ \ for }i=1,\ldots,q. 
\end{equation}
Moreover, the operators $X_{k}^{R}$, with $k>q$, can be expressed in terms of
the $X_{1}^{R},\ldots,X_{q}^{R}$, i.e., for every $k=q+1,\ldots,n$ there are
constants $\vartheta_{i_{1}i_{2}...i_{\omega_{k}}}^{k}$ such that
\begin{equation}     \label{theta}
X_{k}^{R}=\sum_{1\leq i_{j}\leq q}
\vartheta_{i_{1}i_{2}...i_{\omega_{k}}}^{k}X_{i_{1}}^{R}X_{i_{2}}^{R}...
X_{i_{\omega_{k}}}^{R}. 
\end{equation}

\textbf{Convolutions.} The \emph{convolution }of two functions in $\mathbb{G}$
is defined as
\[
(f\ast g)(x)=\int_{{{\mathbb{R}}^{n}}}\text{ }f(x\circ y^{-1})\,g(y)\,dy=
\int_{{{\mathbb{R}}^{n}}}\text{ }g(y^{-1}\circ x)\,f(y)\,dy\text{,}
\]
for every couple of functions for which the above integrals make sense. From
this definition we see that if $P$ is any left invariant differential
operator, then
\[
P(f\ast g)=f\ast Pg
\]
(provided the integrals converge). Note that, if $\mathbb{G}$ is not abelian,
we cannot write $f\ast Pg=Pf\ast g$. Instead, if $X$ and $X^{R}$ are,
respectively, a left invariant and right invariant vector field which agree at
the origin, the following hold (see \cite{St}, p.607)
\[
\left( Xf\right) \ast g=f\ast\left(  X^{R}g\right)  \text{; \ \ }
X^{R}\left(  f\ast g\right)  =\left(  X^{R}f\right)  \ast g.
\]
Explicitly this means
\begin{equation}   \label{X (convoluz)}
\int_{{{\mathbb{R}}^{n}}}\text{ }X^{R}g(y^{-1}\circ x)\,f(y)\,dy=
\int_{{{\mathbb{R}}^{n}}}\text{ }g(y^{-1}\circ x)\,Xf(y)\,dy
\end{equation}
for any $f,g$ for which the above integrals make sense.

\textbf{Heat kernels.} If $\mathbb{G}=({{\mathbb{R}}^{n}},\circ,D(\lambda))$ 
is a Carnot group, we can naturally
define its \emph{parabolic version} setting, in $\mathbb{R}^{n+1}$:
\[
\left(  t,u\right)  \circ_{P}\left(  s,v\right)  =\left(  t+s,u\circ v\right)
;\text{ }D_{P}\left(  \lambda\right)  \left(  t,u\right)  =\left(  \lambda
^{2}t,D\left(  \lambda\right)  u\right)  .
\]
If we define the \emph{parabolic homogeneous group} 
$\mathbb{G}_{P}=\left(\mathbb{R}^{n+1},\circ_{P},D_{P}(\lambda)\right)$, 
its homogeneous dimension
is $Q+2$. Let us now consider the heat operator in $\mathbb{G}_{P},$
\[
\mathcal{H}=\partial_{t}-\sum_{j=1}^{q}X_{j}^{2}=\partial_t-L
\]
which is translation invariant, homogeneous of degree $2$ and hypoelliptic. By
a general result due to Folland, see \cite{FoSt}, $\mathcal{H}$ possesses a
homogeneous fundamental solution $h\left(  t,x\right)  ,$ usually called
\emph{the heat kernel on $\mathbb{G}$}. The next theorem collects several
well-known important facts about $h$, which are useful in the sequel; the
statements (i)-(v) below can be found in \cite[Section 1G]{FoSt}, while the estimates
(\ref{BLU 1})-(\ref{BLU 2}) are contained in \cite[Section IV.4]{VSC}.

\begin{theorem}
\label{heatkernel} There exists a function $h(t,x)$ defined in
$\mathbb{R}^{n+1}$ with the following properties:
\begin{itemize}
\item[(i)] $h \in C^{\infty}\left(  {{\mathbb{R}}^{n+1}}\setminus\left\{
0\right\}  \right)  $;
\item[(ii)] $h\left(  \lambda^{2}t,D\left(  \lambda\right)  x\right)  =
\lambda^{-Q}h\left(  t,x\right)  $ for any $t>0,x\in{{\mathbb{R}}^{n}},\ \lambda>0$;
\item[(iii)] $h\left(  t,x\right)  =0$ for any $t<0,x\in{{\mathbb{R}}^{n}}$;
\item[(iv)] $\displaystyle \int_{{{\mathbb{R}}^{n}}}h\left(  t,x\right)  dx=1$ 
for any $t>0$;
\item[(v)] $h\left(  t,x^{-1}\right)  =h\left(  t,x\right)  $ for any
$t>0,x\in{{\mathbb{R}}^{n}}$.
\end{itemize}
Setting $h_t(x)=h(t,x)$, let us introduce the \emph{heat semigroup}, defined as follows 
for any $f\in L^{1}\left(  {{\mathbb{R}}^{n}}\right)$:
\[
W_{t}f(x)=\int_{{{\mathbb{R}}^{n}}}h\left(  t,y^{-1}\circ x\right)f(y)dy
=(f\ast h_t)(x).
\]
Then, for any $f\in L^{1}\left(  {{\mathbb{R}}^{n}}\right)  $ and $t>0$, we
have $W_{t}f\in C^{\infty}\left(  {{\mathbb{R}}^{n}}\right)  $ and the
function $u\left(  t,x\right)  =W_{t}f(x)$ solves the equation 
$\mathcal{H}u=0$ in $(0,\infty)\times{{\mathbb{R}}^{n}}$. Moreover,
\[
u\left(  t,x\right)  \rightarrow f\left(  x\right)  \text{ strongly in }
L^{1}({{\mathbb{R}}^{n}})\text{ as }t\rightarrow0.
\]
Finally, for every nonnegative integers $j,k$, for every $x\in{{{\mathbb{R}}^{n}}},\,t>0$, 
the following Gaussian bounds hold:
\begin{align}  \label{BLU 1}
\mathbf{c}^{-1}t^{-Q/2}e^{-\left\vert x\right\vert ^{2}/\mathbf{c}^{-1}t}  &
\leq h(t,x)\leq\mathbf{c}t^{-Q/2}e^{-\left\vert x\right\vert ^{2}/\mathbf{c}t}
\\   \label{BLU 2}
\left\vert X_{i_{1}}\cdots X_{i_{j}}(\partial_{t})^{k}\,h\left(  t,x\right)\right\vert  
&  \leq\mathbf{c}(j,k)t^{-(Q+j+2k)/2}e^{-\left\vert x\right\vert^{2}/\mathbf{c}t} 
\end{align}
for $i_{1},i_{2},...,i_{j}\in\left\{  1,2,...,q\right\}  .$ Here
$\mathbf{c}\geq1$ is a constant only depending on $\mathbb{G}$, while
$\mathbf{c}(j,k)$ depends on $\mathbb{G},j,k.$
\end{theorem}

\begin{remark}\label{Remark Gaussian}{\rm 
The Gaussian estimates in this context usually involve the so-called {\em 
homogeneous norm} and are more precise than those we use here, that are
stated in terms of the Euclidean norm. We have stated the Gaussian bounds in the 
present form because we don't need the estimates in their full strength and 
\eqref{BLU 1}, \eqref{BLU 2} follow immediately from the usual estimates In this  
way we avoid the introduction of the homogeneous norm, which
we don't need.  

The bound on the derivatives $X_{i_{1}}\cdots X_{i_{j}}\partial_{t}^{k}\,h\left(  t,x\right)$ 
still holds, in the same form, if the vector fields $X_{i_{1}},\cdots,X_{i_{j}}$ are replaced with any
family of $j$ vector fields, homogeneous of degree one; for instance, we will
apply this bound to derivatives with respect to the right invariant vector
fields $X_{i}^{R}$. Moreover, a homogeneity argument shows that if $a\left(x\right)$ 
is a homogeneous function of degree $j^{\prime},$ then
\[
\left\vert X_{i_{1}}...X_{i_{j}}\partial_{t}^{k}
\left[  a\left(  x\right)h\left(  t,x\right)  \right]  \right\vert 
\leq c\left(  j,k,a\right)t^{-\left(  Q+j-j^{\prime}+2k\right)  /2}
e^{-\left\vert x\right\vert^{2}/\mathbf{c}t}.
\]
}\end{remark}

\subsection{$BV$ functions} \label{BVsubsection}

Let us define the Sobolev spaces $W^{1,p}({{\mathbb{G}}})$, $1\leq p<\infty$,
and the space $BV({\mathbb{G}})$ of functions of bounded variation in
$\mathbb{G}$ and list their main properties. We remark that the definition of
$BV(\mathbb{G})$ goes back to \cite{CapDanGar94The} and refer to
\cite{AmbFusPal00Fun} and to \cite{fraserser3} for more information on the
Euclidean and the subriemannian case, respectively. We start from the Sobolev case.

\begin{definition} \label{defSobolev} 
For $1\leq p<\infty$, we say that $f\in L^{p}({{\mathbb{R}}^{n}})$ belongs to 
the Sobolev space $W^{1,p}({{\mathbb{G}}})$ if there are
$f_{1},\ldots,f_{q}\in L^{p}({{\mathbb{R}}^{n}})$ such that
\[
\int_{{{\mathbb{R}}^{n}}}f(x)X_{i}g(x)dx=-
\int_{{{\mathbb{R}}^{n}}} g(x)f_{i}(x)dx,\qquad i=1,\ldots,q,
\]
for all $g\in C_{0}^{1}({{\mathbb{R}}^{n}})$. In this case, we denote $f_{i}$
as $X_{i}f$ and set
\[
\nabla_{X}f=\left(  X_{1}f,X_{2}f,...,X_{q}f\right)  .
\]
We also let
\[
\left\Vert \nabla_{X}f\right\Vert _{L^{1}\left(  \mathbb{R}^{n}\right)  }
=\int_{\mathbb{R}^{n}}\left\vert \nabla_{X}f\right\vert dx=\int_{\mathbb{R}^{n}}
\sqrt{{\displaystyle\sum\nolimits_{i=1}^{q}}\left\vert X_{i}f\right\vert ^{2}}dx.
\]
\end{definition}

We now consider functions of bounded variation.

\begin{definition}\label{defBV} 
For $f\in L^{1}({{\mathbb{R}}^{n}})$ we say that $f\in BV({\mathbb{G}})$  
if there are finite Radon measures $\mu_i$ such that
\[
\int_{{{\mathbb{R}}^{n}}}f(x)X_{i}g(x)dx=-
\int_{{{\mathbb{R}}^{n}}} g(x)d\mu_{i},\qquad i=1,\ldots,q,
\]
for all $g\in C_{0}^{1}({{\mathbb{R}}^{n}})$. In this case, we denote $\mu_{i}$
by $X_{i}f$ and notice that the total variation of the ${\mathbb R}^q$-valued
measure $D_{\mathbb G}f=(X_1f,\ldots,X_qf)$ is given by 
\begin{equation}  \label{deftotvarG}
\left\vert D_{\mathbb{G}}f\right\vert \left(  {{\mathbb{R}}^{n}}\right)
=\sup\left\{  \int_{{{\mathbb{R}}^{n}}}f(x)\mathrm{div}_{{\mathbb{G}}}g(x)dx:
\ g\in C_{0}^{1}\left(  {{\mathbb{R}}^{n}},{\mathbb{R}}^{q}\right),
\left\Vert g\right\Vert _{\infty}\leq1\right\}
\end{equation}
where
\[
\mathrm{div}_{\mathbb{G}}g(x)=\sum_{i=1}^{q}X_{i}g_{i}(x)
\text{, if }g(x)=\left(  g_{1}(x),\ldots,g_{q}(x)\right)
\]
and $\Vert g\Vert_{\infty}=\sup_{x\in{{\mathbb{R}}^{n}}}|g(x)|.$ 
\end{definition}

With the same proof contained e.g. in \cite[Prop. 3.6]{AmbFusPal00Fun}, it is
possible to show that if $f$ belongs to $BV(\mathbb{G})$ then its total
variation $|D_{\mathbb{G}}f|$ is a finite positive Radon measure and there is
a $|D_{\mathbb{G}}f|$-measurable function $\sigma_{f}:{{\mathbb{R}}^{n}}\rightarrow{\mathbb{R}}^{q}$ 
such that $|\sigma_{f}(x)|=1$ for $|D_{\mathbb{G}}f|$-a.e. $x\in{{\mathbb{R}}^{n}}$ and
\begin{equation}    \label{defsigmaf}
\int_{{{\mathbb{R}}^{n}}}f(x)\mathrm{div}_{\mathbb{G}}g(x)dx=
\int_{{{\mathbb{R}}^{n}}}\langle g,\sigma_{f}\rangle d|D_{\mathbb{G}}f|
\end{equation}
for all $g\in C_{0}^{1}(\mathbb{G},{\mathbb{R}}^{q})$ where, here and in the
following, we denote by $\left\langle \cdot,\cdot\right\rangle $ the usual
inner product in $\mathbb{R}^{q}.$ \\
We denote by $D_{\mathbb{G}}f$ the vector measure $-\sigma_{f}|D_{\mathbb{G}}f|$, 
so that $X_{i}f$ is the measure $(-\sigma_{f})_{i}|D_{\mathbb{G}}f|$ and
the following integration by parts holds true
\begin{equation}   \label{byparts_BV}
\int_{{{\mathbb{R}}^{n}}}f(x)X_{i}g(x)dx = - 
\int_{{{\mathbb{R}}^{n}}}g\left(x\right)d\left( X_{i}f\right) \left( x\right)
\end{equation}
for all $g\in C_{0}^{1}(\mathbb{G})$. Note also that
\begin{equation}  \label{dX_i D_G}
\left\vert X_{i}f\right\vert \left(  \mathbb{R}^{n}\right)  \leq
|D_{\mathbb{G}}f|\left(  \mathbb{R}^{n}\right)  . 
\end{equation}
To visualize the above definitions we note that, whenever $f$ is a smooth
function, by the Euclidean divergence theorem we have:
\[
\sigma_{f} =-\frac{\nabla_{X}f}{\left\vert \nabla_{X}f\right\vert };
\qquad 
d\left\vert D_{\mathbb{G}}f\right\vert \left(  x\right) =
\left\vert \nabla_{X}f\left(  x\right)  \right\vert dx;
\]
(whenever $\left\vert \nabla_{X}f\right\vert \neq0$) and 
\[
d\left(  X_{i}f\right)  \left(  x\right)  =X_{i}f\left(  x\right)  dx
\]
It is clear that $W^{1,1}({{\mathbb{G}}})$ functions are $BV(\mathbb{G})$
functions whose measure gradient is absolutely continuous with respect to
Lebesgue measure. Moreover, since it is the supremum of $L^{1}$-continuous
functionals, the total variation is $L^{1}({{\mathbb{R}}^{n}})$ lower
semicontinuous (see \cite[Theorem 2.17]{fraserser3}), i.e., $f_{k}\rightarrow f$ 
in $L^{1}({{\mathbb{R}}^{n}})$ implies that
\begin{equation}    \label{sup L1}
|D_{\mathbb{G}}f|({{\mathbb{R}}^{n}})\leq\liminf_{k\rightarrow\infty }
|D_{\mathbb{G}}f_{k}|({{\mathbb{R}}^{n}}). 
\end{equation}
Namely:
\[
\int_{{{\mathbb{R}}^{n}}}f_{k}(x)\mathrm{div}_{\mathbb{G}}g(x)dx\leq
|D_{\mathbb{G}}f_{k}|({{\mathbb{R}}^{n}}), \quad\forall 
g\in C_{0}^{\infty}\left(  {{\mathbb{R}}^{n}},{\mathbb{R}}^{q}\right)  ,
\left\Vert g\right\Vert_{\infty}\leq1.
\]
Passing to the liminf as $k\rightarrow\infty$ we get 
\[
\int_{{{\mathbb{R}}^{n}}}f(x)\mathrm{div}_{\mathbb{G}}g(x)dx\leq
\liminf_{k\rightarrow\infty}|D_{\mathbb{G}}f_{k}|({{\mathbb{R}}^{n}})
\]
and taking the supremum over all possible $g$'s we get (\ref{sup L1}).

\begin{definition}[Sets of finite $\mathbb{G}$-perimeter]\label{perimeter}
If $\chi_{E}$ is the characteristic function of a measurable set $E\subset\mathbb{R}^{n}$ 
and $|D_{\mathbb{G}}\chi_{E}|$ is finite, we say that $E$ is a set of finite
$\mathbb{G}$-perimeter and we write $P_{\mathbb{G}}(E)$ instead of
$|D_{\mathbb{G}}\chi_{E}|$, $P_{\mathbb{G}}(E,F)$ instead of 
$|D_{\mathbb{G}}\chi_{E}|(F)$ for $F$ Borel. Also, we call (generalized inward) 
$\mathbb{G}$-normal the $q$-vector 
\[
\nu_{E}\left(  x\right)  =-\sigma_{\chi_{E}}\left(  x\right)  .
\]
\end{definition}
Recall that $\left\vert \nu_{E}\left(  x\right)  \right\vert =1$ for
$P_{\mathbb{G}}(E)$-a.e. $x\in\mathbb{R}^{n}.$ In this case (\ref{defsigmaf})
takes the form 
\begin{equation}   \label{div thm}
\int_{E}\mathrm{div}_{\mathbb{G}}g(x)dx=-
\int_{{{\mathbb{R}}^{n}}}\langle g,\nu_{E}\rangle dP_{\mathbb{G}}(E). 
\end{equation}
Moreover, notice that if $E$ has finite perimeter, then by the isoperimetric inequality
in Carnot groups, see e.g. \cite[Theorem 1.18 and Remark 1.19]{CapDanGar94The}, either 
$E$ or its complement $E^c$ has finite measure. 

\begin{remark}{\rm 
To help the reader to visualize the above definitions, let us
specialize them to the case of a bounded smooth domain $E$, see \cite[Prop. 2..22]{fraserser3}. 
Let $n_{E}$ be the Euclidean unit inner normal at $\partial E,$ and
consider the $q$-vector $v$ whose $i$-th component is defined by
\[
v_i = \sum_{j=1}^{n}q_{i}^{j}\left(x\right)  \left(  n_{E}\right)  _{j}\left(  x\right)
\]
(where the $q_{i}^{j}$ are the coefficients of the vector fields $X_{i},$
defined in (\ref{X})). Then
\[
\int_{E}\mathrm{div}_{\mathbb{G}}g(x)dx=-\int_{\partial E}
\langle g,  v \rangle dH^{n-1}\left(  x\right)
\]
from which we read that in this case
\begin{equation}\label{smooth}
\nu_{E} =\frac{v}{\left\vert v \right\vert }, \qquad
dP_{\mathbb{G}}(E) =\left\vert v \right\vert
d\left(H^{n-1}\res{\partial E}\right)\left(  x\right) 
\end{equation}
at least at those points of the boundary where $v\neq0$ (noncharacteristic points). 
Here $H^{n-1}$ is the Euclidean $(n-1)$-dimensional Hausdorff measure and $\res$ denotes the 
restriction of the measure. 
}\end{remark}

Identity (\ref{X (convoluz)}) can be extended to the case $g\in C^{1}\left({\mathbb{G}}\right)$, 
$f\in BV\left(  {\mathbb{G}}\right)  $ as follows:
\begin{equation}    \label{dX(convoluz)}
\int_{{{\mathbb{R}}^{n}}}\text{ }X_{i}^{R}g(y^{-1}\circ x)\,f(y)\,dy=
\int_{{{\mathbb{R}}^{n}}}\text{ }g(y^{-1}\circ x)\,d\left(  X_{i}f\right)  (y)
\end{equation}
where the last integral is made with respect to the measure defined by
$X_{i}f$. To see this, it is enough to take a sequence $f_{k}\in W^{1,1}({\mathbb{G}})$ 
such that $f_{k}\rightarrow f$ in $L^{1}({{\mathbb{R}}^{n}})$ and $X_{i}f_{k}$ 
is weakly$^{\ast}$ convergent to $X_{i}f$ (the existence of such a sequence is proved 
in \cite{fraserser}), so that
\begin{align*}
\lim_{k\rightarrow\infty}
\int_{{{\mathbb{R}}^{n}}}\text{ }X_{i}^{R} g(y^{-1}\circ x)\,f_{k}(y)\,dy  
&  =\lim_{k\rightarrow\infty}
\int_{{{\mathbb{R}}^{n}}}\text{ }g(y^{-1}\circ x)\,X_{i}f_{k}(y)\,dy
\\
&  =\int_{{{\mathbb{R}}^{n}}}\text{ }g(y^{-1}\circ x)\,d\left(  X_{i}f\right)(y).
\end{align*}
Let us come to some finer properties of $BV$ functions that we need only in
the proof of Theorem~\ref{Main Theorem 2}. We refer to \cite[Theorem 2.3.5]{fraserser} 
for a proof of the following statement.

\begin{proposition} [Coarea formula]\label{coarea} 
If $f\in BV({\mathbb{G}})$ then for a.e.
$\tau\in\mathbb{R}$ the set $E_{\tau}=\{x\in{{\mathbb{R}}^{n}}:\ f(x)>\tau\}$
has finite $\mathbb{G}$-perimeter and
\begin{equation}    \label{coareaformula}
\left\vert D_{\mathbb{G}}f\right\vert ({{\mathbb{R}}^{n}})=
\int_{-\infty}^{+\infty}|D_{\mathbb{G}}\chi_{E_{\tau}}|({{\mathbb{R}}^{n}})d\tau.
\end{equation}
Conversely, if $f\in L^{1}({{\mathbb{R}}^{n}})$ and 
$\int_{-\infty}^{+\infty}|D_{\mathbb{G}}\chi_{E_{\tau}}|({{\mathbb{R}}^{n}})d\tau<+\infty$ 
then $f\in BV(\mathbb{G})$ and equality \eqref{coareaformula} holds. Moreover, if
$g:\mathbb{R}^{n}\rightarrow\mathbb{R}$ is a Borel function, then
\begin{equation}    \label{coareag}
\int_{\mathbb{R}^{n}}g(x)d\left\vert D_{\mathbb{G}}f\right\vert (x)=
\int_{-\infty}^{+\infty}\int_{\mathbb{R}^{n}}g(x)d|D_{\mathbb{G}}\chi_{E_{\tau}}|(x)d\tau. 
\end{equation}
\end{proposition}

Next we have to introduce the reduced boundary. This definition relies on a
particular notion of balls in $\mathbb{G}$. We will denote by $B_{\infty}(x,r)$ 
the $d_{\infty}$-balls of center $x$ and radius $r$, where 
$d_{\infty}\left(  x,y\right)  =d_{\infty}\left(  0,y^{-1}\circ x\right)  $ and
$d_{\infty}\left(  0,x\right)  $ is the homogeneous norm in $\mathbb{G}$ which is 
considered in \cite[Theorem 5.1]{fraserser3}. We do not recall its explicit
analytical definition since we will never need it. The distance
$d_{\infty}$ can be used to define the spherical Hausdorff measures
${\mathcal S}_{\infty}^{k}$ through the usual Carath\'eodory construction,
see e.g. \cite[Section 2.10.2]{Federer}.

\begin{definition}  \label{defboundary}
Let $E\subset{{\mathbb{R}}^{n}}$ be measurable with finite
$\mathbb{G}$-perimeter. We say that $x\in\partial_{\mathbb{G}}^{\ast}E$
(reduced boundary of $E$) if the following conditions hold:
\begin{itemize}
\item[(i)] $P_{\mathbb{G}}\left(  E,B_{\infty}(x,r)\right)  >0$ for all $r>0$;
\item[(ii)] the following limit exists
\[
\lim_{r\rightarrow0}\frac{D_{\mathbb{G}}\chi_{E}(B_{\infty}(x,r))}
{|D_{\mathbb{G}}\chi_{E}|(B_{\infty}(x,r))}
\]
and equals $\nu_{E}(x).$
\item[(iii)] the equality $\left\vert \nu_{E}(x)\right\vert =1$ holds.
\end{itemize}
\end{definition}

In order to state the next result, let us introduce some notation. For
$E\subset{{\mathbb{R}}^{n}}$, $x_{0}\in{{\mathbb{R}}^{n}}$, $r>0$ we consider
the translated and dilated set $E_{r,x_{0}}$ defined as
\[
E_{r,x_{0}}=\{x\in{{\mathbb{R}}^{n}}:\ x_{0}\circ D(r)x\in E\}=D(r^{-1})(x_{0}^{-1}\circ E).
\]
For any vector $\nu\in\mathbb{R}^{q}$ we define the sets
\[
S_{\mathbb{G}}^{+}(\nu)=\left\{  x\in\mathbb{R}^{n}:\ \langle\pi x,\nu
\rangle\geq0\right\}  ,\qquad T_{\mathbb{G}}(\nu)=\left\{  x\in\mathbb{R}^{n}:
\ \langle\pi x,\nu\rangle=0\right\}
\]
where $\pi:\mathbb{R}^{n}\rightarrow\mathbb{R}^{q}$ is the projection which
reads the first $q$ components. Hence the set $T_{\mathbb{G}}(\nu)$ is an
Euclidean ($(n-1)$-dimensional) \textquotedblleft vertical plane\textquotedblright
\ in $\mathbb{R}^{n};$ it is also a subgroup of $\mathbb{G}$, since on the
first $q$ components the Lie group operation is just the Euclidean sum.

We are now in a position to state the following result (see 
\cite[Theorem 3.1]{fraserser3}).

\begin{theorem}   [Blow-up on the reduced boundary]\label{rectifiability}
Let $\mathbb{G}$ be of Step $2$. If $E\subset{{\mathbb{R}}^{n}}$ is a set of finite 
$\mathbb{G}$-perimeter and $x_{0}\in\mathcal{\partial}_{\mathbb{G}}^{\ast}E$ then
\begin{equation}     \label{blow-up1}
\lim_{r\rightarrow0}\chi_{E_{r,x_{0}}}=
\chi_{S_{\mathbb{G}}^{+}(\nu_{E}(x_{0}))}\qquad
\mathrm{in\ \ }L_{\mathrm{loc}}^{1}({{\mathbb{R}}^{n}}).
\end{equation}
(Note that $\nu_{E}(x_{0})$ is pointwise well defined for 
$x_{0}\in\mathcal{\partial}_{\mathbb{G}}^{\ast}E,$ by point (ii) in Definition~\ref{defboundary}). 
Moreover, $P_{\mathbb{G}}\left(  E\right)  $-a.e.
$x\in{{\mathbb{R}}^{n}}$ belongs to $\mathcal{\partial}_{\mathbb{G}}^{\ast}E.$
\end{theorem}

The last statement in the above theorem allows us to rewrite
formula \eqref{div thm} as an integral on the reduced
boundary with respect to the $Q-1$ spherical Hausdorff measure as follows:
\begin{equation}\label{div thm boundary}
\int_{E}\mathrm{div}_{\mathbb{G}}g(x)dx=-\theta_\infty\int_{\partial_{\mathbb{G}}^{\ast}E}\langle
g,\nu_{E}\rangle d{\mathcal S}_{\infty}^{Q-1},
\end{equation}
(here $\theta_{\infty}$ is a constant depending on ${\mathbb{G}}$, see
\cite[Theorem 3.10]{fraserser3}), which looks much closer to the classical divergence theorem.
Moreover, the following analogue of \eqref{smooth} holds:
\begin{equation}    \label{perim S_infty}
P_{\mathbb{G}}(E,\cdot) = \theta_{\infty} \mathcal{S}_{\infty}^{Q-1}\res
\left(\mathcal{\partial}_{\mathbb{G}}^{\ast}E\right)  . 
\end{equation}

In order to apply the $L_{\mathrm{loc}}^{1}$ convergence of (\ref{blow-up1})
we will need also the following

\begin{remark}       \label{L1_loc L1}{\rm 
If $E,\left\{  E_{k}\right\}  _{k=1}^{\infty}\subset\mathbb{R}^{n}$ are measurable,  
$\chi_{E_{k}}\rightarrow\chi_{E}$ in 
$L_{\mathrm{loc}}^{1}({{\mathbb{R}}^{n}})$ and 
$f\in L^{1}\left(  \mathbb{R}^{n}\right)  \cap L^{\infty}\left(  \mathbb{R}^{n}\right)$, 
then $f\chi_{E_{k}}\rightarrow f\chi_{E}$ in $L^{1}({{\mathbb{R}}^{n}}).$ In fact,
given $\varepsilon>0$, there is a compact set $K$ such that
\[
\int_{{\mathbb R}^n\setminus K}|f|dx < \varepsilon,
\]
whence
\[
\int_{{{\mathbb{R}}^{n}}}\left\vert \left(  \chi_{E_{k}}-\chi_{E}\right)f\right\vert dx
\leq \varepsilon + 
\left\Vert f\right\Vert_{\infty}\int_{K}\left\vert \chi_{E_{k}}-\chi_{E}\right\vert dx 
\]
and the last integral tends to 0 as $k\to\infty$.
}\end{remark}

\subsection{Main results}

We state here the main results of this paper, namely Theorems~\ref{Main Theorem 1} 
and \ref{Main Theorem 2}.

\begin{theorem} \label{Main Theorem 1} 
Let $f\in L^{1}\left(  {{\mathbb{R}}^{n}}\right)  $. The quantity
\[
\limsup_{t\rightarrow0}
\left\Vert \nabla_{X}W_{t}f\right\Vert _{L^{1}\left({{\mathbb{R}}^{n}}\right)  }
\]
is finite if and only if $f\in BV\left(  \mathbb{G}\right)  $. In this case
the following holds:
\begin{align*}
|D_{\mathbb{G}}f|\left(  {{\mathbb{R}}^{n}}\right)  \leq &  
\liminf_{t\rightarrow0}\left\Vert \nabla_{X}W_{t}f\right\Vert _{L^{1}
\left({{\mathbb{R}}^{n}}\right)  } \leq\limsup_{t\rightarrow0}
\left\Vert \nabla_{X}W_{t}f\right\Vert_{L^{1}\left(  {{\mathbb{R}}^{n}}\right)  }\\
\leq &  (1+c)\left\vert D_{\mathbb{G}}f\right\vert \left(  {{\mathbb{R}}^{n}}\right)  ,
\end{align*}
where $c\geq0$ is a constant depending only on ${\mathbb{G}}$.
\end{theorem}

Theorem~\ref{Main Theorem 1} shows that even in Carnot groups Definition~\ref{defBV} 
is equivalent to the analogue of De Giorgi's original definition
\eqref{variation3} given in the Euclidean case, even though we don't know if
the limit exists. Section~\ref{proof1} is devoted to the proof of Theorem~\ref{Main Theorem 1}.

The next result requires the additional hypothesis that $\mathbb{G}$ is a
Carnot group of Step 2. As we explain in Section~\ref{proof2}, the reason is
that our proof uses the rectifiability of the reduced boundary of a set of
finite perimeter in $\mathbb{G}$, which is known only in the Step 2 case, see
\cite{fraserser3}. Our argument works whenever the rectifiability theorem is
true, and then would extend immediately to all cases where the 
rectifiability result were extended.

We start by defining the function
\begin{equation}    \label{phiG}
\phi_{\mathbb{G}}(\nu)=\int_{T_{\mathbb{G}}(\nu)}h(1,x)dx, 
\end{equation}
for vectors $\nu\in\mathbb{R}^{q}$. Here the integral is taken over
$T_{\mathbb{G}}(\nu)$, which is a hyperplane in $\mathbb{R}^{n};$ to simplify
notation, we keep denoting by $dx$ the $\left(  n-1\right)  $-dimensional
Lebegue measure over $T_{\mathbb{G}}(\nu)$.

The function $\phi_{\mathbb{G}}$ is continuous and, using the Gaussian bounds
\eqref{BLU 1}, there are constants $c_{1},c_{2}$, depending on the group, such
that
\begin{equation} \label{phi bound}
0<c_{1}\leq\phi_{\mathbb{G}}(\nu)\leq c_{2}.
\end{equation}

\begin{remark} [Rotation invariance of heat kernels]\label{Remark rotation}{\rm 
Notice that if the heat kernel is invariant under horizontal rotations (that is,
under Euclidean rotations in the space $\mathbb{R}^{q}$ of the first variables
$x_{1},x_{2},...,x_{q}$ of $\mathbb{R}^{n}$), then $\phi_{\mathbb{G}}$ is
actually a constant (that is, independent of $\nu$). This happens for instance
in all groups of Heisenberg type (briefly called H-type groups), in view of
the known formula assigning the heat kernel in that context (for a discussion
of H-type groups, see for instance \cite[Chap. 18]{BLUbook}; for the
computation of the heat kernel in this context, see \cite{R}). By the way, we
point out that in the Heisenberg groups with more than one vertical direction the heat kernel is
invariant for horizontal rotations, but the sublaplacian is not. This can 
be seen through a direct computation based on (3.14) in \cite{BLUbook}. On
the other hand, from the general formula assigning the heat kernel in Carnot
groups of step two, proved by \cite{Cy} (see also \cite{BGG}), one can expect
the existence of groups of step two (more general than H-type groups) where
the heat kernel is not invariant under horizontal rotations and the function
$\phi_{\mathbb{G}}$ is not constant. 
}\end{remark}

Let us first state our next result in the case of perimeters.

\begin{theorem} \label{diffusionperimeters}
Let $\mathbb{G}$ be of Step 2. If $E\subset {{\mathbb{R}}^{n}}$ is a set of finite 
$\mathbb{G}$-perimeter, the following equality holds
\begin{equation}    \label{mainequality}
\lim_{t\rightarrow0}\frac{1}{2\theta_{\infty}\sqrt{t}}\int_{E^{c}}W_{t}\chi_{E}(x)dx=
\int_{\mathcal{\partial}_{\mathbb{G}}^{\ast}E}\phi_{\mathbb{G}}(\nu_{E}) 
d{\mathcal{S}}_{\infty}^{Q-1} .
\end{equation}
Conversely, if either $E$ or $E^c$ has finite measure and 
\[
\liminf_{t\rightarrow0}\frac{1}{\sqrt{t}}\int_{E^{c}}W_{t}\chi_{E}(x)\ dx<+\infty ,
\]
then $E$ has finite perimeter, and \eqref{mainequality} holds.
\end{theorem}

For general $BV$ functions, we have the following.

\begin{theorem}    \label{Main Theorem 2} 
Let $\mathbb{G}$ be of Step 2. Then, if $f\in BV({\mathbb{G}})$ the following equality holds:
\begin{equation}     \label{diffusionbv}
\int_{{\mathbb{R}}^{n}}\phi_{\mathbb{G}}(\sigma_{f})d\left\vert D_{\mathbb{G}}f\right\vert 
=\lim_{t\rightarrow0}\frac{1}{4\sqrt{t}}
\int_{{{\mathbb{R}}^{n}}\times{{\mathbb{R}}^{n}}}|f(x)-f(y)|h(t,y^{-1}\circ x)dxdy,
\end{equation}
where $\sigma_{f}$ is defined in \eqref{defsigmaf}. Conversely, if $f\in L^1({\mathbb R}^n)$ 
and the liminf of the quantity on the right hand side is finite, then $f\in BV(\mathbb{G})$
and \eqref{diffusionbv} holds.
\end{theorem}

\begin{remark}     \label{Remark H-type}{\rm 
Note that, in view of (\ref{perim S_infty}),
(\ref{phi bound}) and Remark~\ref{Remark rotation}, the right-hand side of
(\ref{mainequality}) is always equivalent to, and in groups of Heisenberg-type
is a multiple of, $P_{\mathbb{G}}\left(  E\right)  $; the left-hand side of
(\ref{diffusionbv}) is always equivalent to, and in groups of Heisenberg-type
is a multiple of, $\left\vert D_{\mathbb{G}}f\right\vert \left(\mathbb{R}^{n}\right)$. 
}\end{remark}

The proofs of the above results will be given in Section~\ref{proof2}.

\section{Proof of Theorem~\ref{Main Theorem 1}}

\label{proof1}

In this section we give the proof of Theorem~\ref{Main Theorem 1}; first of
all, we notice that $W_{t}f\rightarrow f$ in $L^{1}$ as $t\rightarrow0;$
moreover $W_{t}f\in C^{\infty}\left(  \mathbb{R}^{n}\right)  $, hence
\[
|D_{\mathbb{G}}W_{t}f|({\mathbb{R}^{n}})=\int_{{{\mathbb{R}}^{n}}}|\nabla
_{X}W_{t}f(x)|dx.
\]
Therefore by the lower semicontinuity of the variation, that is (\ref{sup L1}), 
the inequality
\begin{equation}   \label{lowersc}
|D_{\mathbb{G}}f|({\mathbb{R}^{n}})\leq\liminf_{t\rightarrow0}
\int_{{{\mathbb{R}}^{n}}}|\nabla_{X}W_{t}f(x)|dx
\end{equation}
holds, and then it remains to prove only the upper bound. We also notice that
inequality (\ref{lowersc}) implies that if $f\in L^{1}({{\mathbb{R}}^{n}})$ is
such that the right hand side of (\ref{lowersc}) is finite, then $f\in BV({\mathbb{G}})$. 
So, Theorem~\ref{Main Theorem 1} will follow if we prove that
\begin{equation}   \label{limsupbv}
\limsup_{t\rightarrow0}\int_{{{\mathbb{R}}^{n}}}|\nabla_{X}W_{t}f(x)|dx
\leq\left(  1+c\right)  |D_{\mathbb{G}}f|({\mathbb{R}^{n}})
\end{equation}
for all $f\in BV(\mathbb{G})$.

We start with the following result, which will be useful in the proof of
Theorem~\ref{Main Theorem 1}. It is a consequence of the algebraic properties
of $\mathbb{G}$ and the Gaussian estimates on the heat kernel $h$ (see Section~\ref{sec carnot}).

\begin{lemma}   \label{lemmanuclei} 
For $i,j\in\{1,\ldots,q\}$, let $G_{j}^{i}$ be the kernel defined by the identity:
\begin{align}
\sum_{k=q+1}^{n}X_{k}^{R}(c_{i}^{k}(\cdot)h(t,\cdot))(z)
&  =\sum_{j=1}^{q}X_{j}^{R}\sum_{k=q+1}^{n}\sum_{1\leq i_{l}\leq q}
\vartheta_{ji_{2}\ldots i_{\omega_{k}}}^{k}X_{i_{2}}^{R}...X_{i_{\omega_{k}}}^{R}\,
\left(  c_{i}^{k}(\cdot)\,h\left(  t,\cdot\right)  \right)(z)
\nonumber\label{defgij}
\\
&  =\sum_{j=1}^{q}X_{j}^{R}G_{j}^{i}(t,z),
\end{align}
where the functions $c_{i}^{k}$ and the constants $\vartheta_{i_{1}i_{2}...i_{\omega_{k}}}^{k}$ 
are those introduced in \eqref{Xi_XiR} and \eqref{theta}, respectively. 
They have the following properties:

\begin{itemize}
\item[(i)] $G_{j}^{i}\left(  \lambda^{2}t,D\left(  \lambda\right)  z\right)  =
\lambda^{-Q}G_{j}^{i}\left(  t,z\right)  $ for any $\lambda>0,t>0,z\in{{\mathbb{R}}^{n}}$;
\item[(ii)] there is a positive constant $c$, independent of $t$, such that
\[
\int_{{{\mathbb{R}}^{n}}}\left\vert G_{j}^{i}\left(  t,z\right)  \right\vert dz\leq c;
\]
\item[(iii)] ${\displaystyle \int_{{{\mathbb{R}}^{n}}}G_{j}^{i}\left(  t,z\right)  dz=0
\text{ for any }t>0;}$
\item[(iv)] for every $\varepsilon>0$, $t_{0}>0$ there exists $R>0$ such that
\[
\int_{{{\mathbb{R}}^{n}}\setminus B_{R}} |G_{j}^{i}(t,z)|dz < \varepsilon
\text{ for any }0<t\leq t_{0}.
\]
\end{itemize}
\end{lemma}
\begin{proof}
(i) holds because $c_{i}^{k}(\cdot)$ is homogeneous of degree $\omega_{k}-\omega_{i}=\omega_{k}-1$ 
(for $i=1,2,...,q$), $h(t,\cdot)$ is homogeneous of degree $-Q,$ and 
$X_{i_{2}}^{R}...X_{i_{\omega_{k}}}^{R}$ is a homogeneous differential operator of degree 
$\omega_{k}-1.$ For the same reason, the Gaussian estimates on $h$ also imply, by 
Remark~\ref{Remark Gaussian},
\begin{equation} \label{Gaussian G}
|G_{j}^{i}(t,z)|=\left\vert X_{i_{2}}^{R}\ldots X_{i_{\omega_{k}}}^{R}\,
\left(  c_{i}^{k}(\cdot)\,h\left(  t,\cdot\right)  \right)(z)\right\vert 
\leq\mathbf{c}t^{-Q/2}e^{-\left\vert z\right\vert^{2}/\mathbf{c}t}
\end{equation}
To prove (ii), with the change of variables $z=D\left(  \sqrt{t}\right)  w$
and using (i) and (\ref{Gaussian G}), we compute:
\begin{align}
\int_{\mathbb{R}^{n}}\left\vert G_{j}^{i}\left(  t,z\right)  \right\vert dz
&  =\int_{\mathbb{R}^{n}}
\left\vert G_{j}^{i}\left(  t,D\left(  \sqrt{t}\right)  w\right)  \right\vert t^{Q/2}dw
\nonumber\\    \label{2.9}
&  =\int_{\mathbb{R}^{n}}\left\vert G_{j}^{i}\left(  1,w\right)  \right\vert dw
\leq\int_{\mathbb{R}^{n}}\mathbf{c}e^{-\left\vert w\right\vert^{2}/\mathbf{c}}dw
\end{align}
which is a finite constant. (iii)\ holds because $G_{j}^{i}$ is, by
definition, a linear combination of derivatives of the kind
\[
X_{i_{2}}^{R}\left[  X_{i_{3}}^{R}...X_{i_{\omega_{k}}}^{R}\,
\left(  c_{i}^{k}(\cdot)\,h\left(  t,\cdot\right)  \right)  \right]  =
X_{i_{2}}^{R}H\left(t,\cdot\right),
\]
we know that $\left(  X_{i}^{R}\right)  ^{T}=-X_{i}^{R}$ (where $\left({}\right)^{T}$ 
denotes transposition) and therefore we may deduce that 
$\int_{\mathbb{R}^{n}}X_{i}^{R}f\left(  x\right)  dx=0$ 
for any $f$ for which this integral makes sense.
To prove (iv), it suffices to modify the computation (\ref{2.9}) as follows:
\[
\int_{\mathbb{R}^{n}\setminus B_{R}}\left\vert G_{j}^{i}\left(  t,z\right)
\right\vert dz\leq\int_{\left\vert w\right\vert >R/\sqrt{t}}\mathbf{c}%
e^{-\left\vert w\right\vert ^{2}/\mathbf{c}}dw<\varepsilon
\]
for $R$ large enough and $t\leq t_{0}$.
\end{proof}

The following result contains the main estimate on the commutator between the
derivative $X_{i}$ and the heat semigroup 
$X_{i}\left(  W_{t}f\right)-W_{t}\left(  X_{i}f\right)  $.

\begin{proposition}    \label{Prop XW-WX} 
For any $f\in BV\left(  \mathbb{G}\right)  $ and $i\in\{1,2,\ldots,q\}$, $t>0$, 
there exists an $L^{1}$ function $\mu_{t}^{i}$ on ${{\mathbb{R}}^{n}}$ such that
\[
X_{i}\left(  W_{t}f\right)  (x)=W_{t}\left(  X_{i}f\right)  (x)+\mu_{t}^{i}(x),
\]
where $W_{t}\left(  X_{i}f\right)  $ is the function defined by
\[
W_{t}\left(  X_{i}f\right)  (x):=\int_{{{\mathbb{R}}^{n}}}
h(t,y^{-1}\circ x)dX_{i}f(y).
\]
Moreover, for any $t>0,$
\begin{equation}   \label{estmu}
\Vert\mu_{t}^{i}\Vert_{L^{1}({{\mathbb{R}}^{n}})}\leq cq|D_{\mathbb{G}}f|({{\mathbb{R}}^{n}}),
\end{equation}
where $c$ is the constant in Lemma~\ref{lemmanuclei}, (ii).
\end{proposition}
\begin{proof}
Let $f\in BV\left(  \mathbb{G}\right)  $ and fix $i\in\{1,2,...,q\}$. By
(\ref{Xi_XiR}) and (\ref{dX(convoluz)}) we have
\begin{align}  \label{nuclei G} 
X_{i}W_{t}f\left(  x\right) = &
\int_{\mathbb{R}^{n}}X_{i}h\left(t,y^{-1}\circ x\right)  f\left(  y\right)  dy
\\      
= & \int_{\mathbb{R}^{n}}X_{i}^{R}h\left(  t,y^{-1}\circ x\right)f\left(y\right)dy
+\int_{\mathbb{R}^{n}}\sum_{k=q+1}^{n}X_{k}^{R}\,\left(  c_{i}^{k}(\cdot)\,
h\left(  t,\cdot\right)  \right)  \left(  y^{-1}\circ x\right)f\left(  y\right)  dy
\nonumber\\
= & \int_{\mathbb{R}^{n}}X_{i}^{R}h\left(  t,y^{-1}\circ x\right)  f\left(y\right)  dy
\nonumber\\  
&  +\int_{\mathbb{R}^{n}}\sum_{k=q+1}^{n}\sum_{1\leq i_{j}\leq q}
\vartheta_{i_{1}i_{2}...i_{\omega_{k}}}^{k}X_{i_{1}}^{R}X_{i_{2}}^{R}...X_{i_{\omega_{k}}}^{R}
\,\left(  c_{i}^{k}(\cdot)\,h\left(t,\cdot\right)  \right)  \left(  y^{-1}\circ x\right)  
f\left(  y\right) dy
\nonumber\\
= & \int_{\mathbb{R}^{n}}h\left(  t,y^{-1}\circ x\right)  d\left(X_{i}f\right)  \left(  y\right)  
\nonumber \\
&  +\sum_{k=q+1}^{n}\sum_{1\leq i_{j}\leq q}\vartheta_{i_{1}i_{2}
...i_{\omega_{k}}}^{k}\int_{\mathbb{R}^{n}}X_{i_{2}}^{R}...X_{i_{\omega_{k}}}^{R}\,
\left(  c_{i}^{k}(\cdot)\,h\left(  t,\cdot\right)  \right)  \left(y^{-1}\circ x\right)  
d\left(  X_{i_{1}}f\right)  \left(  y\right) 
\nonumber\\
= &  W_{t}\left(  X_{i}f\right)  \left(  x\right)  +\sum_{j=1}^{q}
\int_{\mathbb{R}^{n}}G_{j}^{i}\left(  t,y^{-1}\circ x\right)  
d\left(X_{j}f\right)  \left(  y\right)  ,\nonumber
\end{align}
where the kernels $G_{j}^{i}$ are defined in Lemma~\ref{lemmanuclei}. Set
\begin{equation}
\mu_{t}^{i}\left(  x\right)  =\sum_{j=1}^{q}\int_{\mathbb{R}^{n}}G_{j}%
^{i}\left(  t,y^{-1}\circ x\right)  d\left(  X_{j}f\right)  (y), \label{mu_t}%
\end{equation}
then by Lemma~\ref{lemmanuclei} and (\ref{dX_i D_G}) we have
\[
\left\Vert \mu_{t}^{i}\right\Vert _{L^{1}\left(  \mathbb{R}^{n}\right)  }\leq
c\sum_{j=1}^{q}\int_{\mathbb{R}^{n}}\left\vert d\left(  X_{j}f\right)\right\vert (y)
\leq cq\left\vert D_{\mathbb{G}}f\right\vert \left(\mathbb{R}^{n}\right)
\]
that is (\ref{estmu}). By (\ref{nuclei G}) and (\ref{mu_t}), we can write:
\begin{equation}    \label{limsup}
X_{i}\left(  W_{t}f\right)  (x)=W_{t}\left(  X_{i}f\right)  (x)+\mu_{t}%
^{i}(x).
\end{equation}
\end{proof}

\begin{proof} [Proof of Theorem~\ref{Main Theorem 1}]
Let us conclude with the proof of \eqref{limsupbv}. By Proposition~\ref{Prop XW-WX} 
we have (with the supremum taken, accordingly to Definition~\ref{defBV}, 
on all the functions $g\in C_{0}^{1}(\mathbb{R}^{n},{\mathbb{R}}^{q})$ 
such that $\Vert g\Vert_{\infty}\leq1$):
\begin{align*}
& \int_{\mathbb{R}^{n}}|\nabla_{X}W_{t}f(x)|dx
=\sup_{g}\left\{\int_{\mathbb{R}^{n}}\sum_{i=1}^{q}g_{i}(x)X_{i}(W_{t}f)(x)dx\right\} 
\\
& =\sup_{g}\left\{  \int_{\mathbb{R}^{n}}\sum_{i=1}^{q}g_{i}\left(  x\right)
\left(  \int_{\mathbb{R}^{n}}h(t,y^{-1}\circ x)dX_{i}f(y)\right)dx
+\int_{\mathbb{R}^{n}}\sum_{i=1}^{q}g_{i}(x)\mu_{t}^{i}(x)dx\right\}
\end{align*}
since $h(t,z^{-1})=h(t,z)$
\begin{align*}
&  =\sup_{g}\left\{  \sum_{i=1}^{q}\int_{\mathbb{R}^{n}}\left(  \int_{\mathbb{R}^{n}}
h(t,x^{-1}\circ y)g_{i}\left(  x\right)  dx\right)dX_{i}f(y)
+\int_{\mathbb{R}^{n}}\sum_{i=1}^{q}g_{i}(x)\mu_{t}^{i}(x)dx\right\}
\\
&  =\sup_{g}\left\{  \sum_{i=1}^{q}\int_{\mathbb{R}^{n}}W_{t}g_{i}(y)dX_{i}f(y)
+\int_{\mathbb{R}^{n}}\sum_{i=1}^{q}g_{i}(x)\mu_{t}^{i}(x)dx\right\} 
\\
&  =\sup_{g}\left\{ -\sum_{i=1}^{q}\int_{\mathbb{R}^{n}}X_{i}(W_{t}g_{i})(y)f(y)dy
+\int_{\mathbb{R}^{n}}\sum_{i=1}^{q}g_{i}(x)\mu_{t}^{i}(x)dx\right\}
\end{align*}
where in the last identity we have used (\ref{byparts_BV}). We now exploit the
fact that the supremum on all compactly supported functions $g$ in the above expression can be replaced by 
the supremum on functions $\phi$ rapidly decreasing at infinity (as $W_{t}g$ is, by the
Gaussian estimates) verifying the constraint $\Vert\phi\Vert_{\infty}\leq1$; 
therefore the last expression is, by (\ref{estmu})
\[
\leq|D_{\mathbb{G}}f|+\sum_{i=1}^q\Vert\mu_{t}^{i}\Vert_{L^{1}(\mathbb{R}^{n})}
\leq(1+cq^2)\left\vert D_{\mathbb{G}}f\right\vert (\mathbb{R}^{n}).
\]
We have therefore proved that for any $t>0$, 
\[
\int_{\mathbb{R}^{n}}|\nabla_{X}W_{t}f(x)|dx\leq
(1+cq^2)\left\vert D_{\mathbb{G}}f\right\vert (\mathbb{R}^{n}).
\]
Passing to the limsup as $t\rightarrow0$ we are done.
\end{proof}

\section{Proof of Theorem~\ref{Main Theorem 2}}    \label{proof2}

Theorem~\ref{Main Theorem 2}, thanks to coarea formula, follows from Theorem~\ref{diffusionperimeters}, 
so we first prove the latter.

As a preliminary result, we present a characterization of $BV({\mathbb{G}})$
functions analogous to that in \cite[Proposition 2.3]{Cap97Reg}.

\begin{lemma}
\label{lemmaCap} 
If $f\in L^{1}({{\mathbb{R}}^{n}})$, $z\in{{\mathbb{R}}^{n}}$
and 
\begin{equation}\label{Zincrement}
\liminf_{t\to 0}\frac{1}{t} \int_{{\mathbb{R}}^{n}}|f(x\circ D(t)z)-f(x)|dx <\infty ,
\end{equation}
then the distributional derivative of $f$ along $Z=\sum_{j=1}^{q}z_{j}X_{j}$ is a finite measure.
\end{lemma}
\begin{proof}
The proof is essentially the same contained in \cite{Cap97Reg}, but we repeat
it for the reader's convenience.  We start by recalling the following identity 
(see \cite[p. 17]{BLUbook}): if we take $\phi\in C_{c}^{1}({{\mathbb{R}}^{n}})$ and $Z$
is any vector field which at the origin equals $z$, then
\[
Z\phi\left(  x\right)  =\frac{d}{dt}\left[  \phi\left(  x\circ tz\right)\right]_{/t=0}.
\]
Let us apply this identity to the vector field $Z=\sum_{i=1}^{q}z_{i}X_{i}$
which at the origin takes the value
\[
z^{\ast}=\left(  z_{1},z_{2},...,z_{q},0,0,...,0\right)  .
\]
For any $z\in\mathbb{R}^{n},$
\[
\frac{d}{dt}\left[\phi\left(x\circ D(t)z\right)\right]_{/t=0}
=\frac{d}{dt}\left[  \phi\left(  x\circ tz^{\ast}\right)  \right]_{/t=0}
\]
hence
\[
\frac{d}{dt}\left[\phi\left(x\circ D(t)z\right)\right]_{/t=0}
=\sum_{i=1}^{q}z_{i}\left(  X_{i}\phi\right)  \left(  x\right)  .
\]
Since 
\begin{align*}
&\left|\int_{{\mathbb{R}}^{n}}f(x)\frac{\phi(x\circ D(t)z^{-1})-\phi(x)}{t}dx\right|
=\left|\int_{{\mathbb{R}}^{n}}\frac{f(x\circ D(t)z)-f(x)}{t}\phi(x)dx\right|
\\
&\leq \|\phi\|_\infty \frac{1}{|t|} \int_{{\mathbb{R}}^{n}}|f(x\circ D(t)z)-f(x)|dx ,
\end{align*}
from \eqref{Zincrement} we deduce that the functional
\[
T_{f}\phi: 
= \lim_{t\rightarrow0}\int_{{\mathbb{R}}^{n}}f(x)\frac{\phi(x\circ D(t)z^{-1})-\phi(x)}{t}dx
=-\int_{{\mathbb{R}}^{n}}f(x)Z\phi(x)dx
\]
is well defined and satisfies the condition $|T_{f}\phi|\leq C\Vert\phi\Vert_{\infty}$. 
Then, there exists a measure $\mu_{Z}$ such that
\[
T_{f}\phi=\int_{{\mathbb{R}}^{n}}\phi(x)d\mu_{Z}(x),\qquad\forall\phi\in C_{c}^{1}({{\mathbb{R}}^{n}}).
\]
By the density of $C_{c}^{1}({{\mathbb{R}}^{n}})$ in $C_{c}({{\mathbb{R}}^{n}})$, we get the conclusion.
\end{proof}

\begin{proof}  [Proof of Theorem~\ref{diffusionperimeters}]
Let us introduce the functions
\begin{align*}
g(t) &  =\int_{E^{c}}W_{t}\chi_{E}(x)dx
\\
F(t) &  =g\left(  t^{2}\right)  =\int_{E^{c}}W_{t^{2}}\chi_{E}(x)dx.
\end{align*}
The statement we want to prove is: 
\[
\lim_{t\rightarrow0}\frac{g\left(  t\right)  }{\sqrt{t}}=2\theta_\infty
\int_{\partial_{\mathbb{G}}^{\ast}E}\phi_{\mathbb{G}}(\nu_{E})d{\mathcal{S}}_{\infty}^{Q-1}.
\]
First, notice that $g(t)\to\int_{E^{c}}\chi_{E}(x)dx=0$ as $t\downarrow 0$. Indeed, 
\begin{equation}  \label{decomp}
W_{t}\chi _{E}-\chi _{E}=-\left( W_{t}\chi _{E^{c}}-\chi _{E^{c}}\right) ,
\end{equation} 
since 
\[
\chi _{E}+\chi _{E^{c}}=1=W_{t}1=W_{t}\chi _{E}+W_{t}\chi _{E^{c}}
\]
and, since by the finiteness of the 
perimeter of $E$ either $E$ or $E^c$ has finite measure, 
$\|W_t\chi_E-\chi_E\|_{L^1({\mathbb R}^n)}\to 0$ as $t\downarrow 0$.
By De L'Hospital rule 
we can evaluate $\lim_{t\rightarrow 0}g\left(  t\right)  /\sqrt{t}$ as
\[
\lim_{t\rightarrow0}2\sqrt{t}g^{\prime}\left(  t\right)  
=\lim_{t\rightarrow 0}2tg^{\prime}\left(  t^{2}\right)  
=\lim_{t\rightarrow0}F^{\prime}\left( t\right)  .
\]
Computing the derivative of $F$ using the divergence theorem
(\ref{div thm boundary}) and the properties of $h(t,\cdot)$, we get:
\begin{align*}
F^{\prime}(t) &  = 2t\int_{E^{c}}LW_{t^{2}}\chi_{E}dx
=2t\int_{E^{c}}\mathrm{div}_{\mathbb{G}}\,\nabla_{X}W_{t^{2}}\chi_{E}dx
\\
& = 2\theta_{\infty}t\int_{\partial_{\mathbb{G}}^{\ast}E}
\langle\nabla_{X}W_{t^{2}}\chi_{E},\nu_{E}\rangle d{\mathcal{S}}_{\infty}^{Q-1}
\\
& =2\theta_{\infty}t\int_{\partial_{\mathbb{G}}^{\ast}E}
\langle\nabla_{X}\left(  \int_{E}h(t^{2},y^{-1}\circ x)dy\right),\nu_{E}(x)\rangle
d{\mathcal{S}}_{\infty}^{Q-1}(x).
\end{align*}
Now, taking into account the fact that $h\left(t,z^{-1}\right) =h\left(t,z\right)$, 
we have 
\begin{align*}
\int_{E}h(t^{2},y^{-1}\circ x)dy &  
=\int_{E}h(t^{2},x^{-1}\circ y)dy
=\int_{E}t^{-Q}h\left( 1,D\left(\frac{1}{t}\right)\left(x^{-1}\circ y\right)\right)dy
\\
& =\int_{E}t^{-Q}h\left(1,D\left(\frac{1}{t}\right)\left(x^{-1}\right)
\circ D\left(\frac{1}{t}\right)\left(y\right)  \right)dy 
\\
& =\int_{D\left(\frac{1}{t}\right)E}h\left(1,D\left(\frac{1}{t}\right)\left(x^{-1}\right)  
\circ z\right)  dz
\\
&  =\int_{D\left(  \frac{1}{t}\right)  E}
h\left(  1,z^{-1}\circ D\left(\frac{1}{t}\right)  \left(  x\right)  \right)  dz
\end{align*}
hence
\begin{align*}
F^{\prime}(t) &  =
2\theta_{\infty}t\int_{\partial_{\mathbb{G}}^{\ast}E}\!\!\left\langle 
\frac{1}{t}\int_{D\left(  \frac{1}{t}\right)  E}\nabla_{X}
h\left(  1,z^{-1}\circ D\left(  \frac{1}{t}\right)  \left(  x\right)\right)  dz,
\nu_{E}(x)\right\rangle d{\mathcal{S}}_{\infty}^{Q-1}(x)
\\
&  =2\theta_{\infty}\int_{\partial_{\mathbb{G}}^{\ast}E}
\left\langle\int_{D\left(  \frac{1}{t}\right)  E}\nabla_{X}
h(1,z^{-1}\circ D\left(\frac{1}{t}\right)  \left(  x\right)  dz,
\nu_{E}(x)\right\rangle d{\mathcal{S}}_{\infty}^{Q-1}(x)
\\
&  =2\theta_{\infty}\int_{\partial_{\mathbb{G}}^{\ast}E}
\left\langle \int_{D(1/t)(x^{-1}E)}\nabla_{X}h(1,z^{-1})dz,\nu_{E}(x)\right\rangle
d{\mathcal{S}}_{\infty}^{Q-1}(x)
\\
&  =2\theta_{\infty}\sum_{i=1}^{q}\int_{\partial_{\mathbb{G}}^{\ast}E}\nu_{E}(x)_{i}
\left(  \int_{D(1/t)(x^{-1}E)}X_{i}h(1,z^{-1})dz\right)d{\mathcal{S}}_{\infty}^{Q-1}(x).
\end{align*}
By Theorem~\ref{rectifiability} (blow-up of the reduced boundary), we have the
$L^1_{\rm loc}$ convergence
\begin{equation}\label{blow up}
E_{1/t,x^{-1}}=D(1/t)(x^{-1}E)\rightarrow 
S_{\mathbb{G}}^{+}(\nu_{E}(x)).
\end{equation}
Since by the Gaussian estimates $X_{i}h(1,z^{-1})\in L^{1}\cap C^{\infty}(\mathbb{R}^{n})$,  
the integral on $E_{1/t,x^{-1}}$ converges to the integral on $S_{\mathbb{G}}^{+}(\nu_{E}(x))$
(see Remark~\ref{L1_loc L1}), is a continuous function of $x$ and we can apply the dominated 
convergence theorem, getting 
\begin{align*}
\lim_{t\rightarrow0}F^{\prime}\left(  t\right)   &  =2\theta_{\infty}
\int_{\partial_{\mathbb{G}}^{\ast}E}\sum_{i=1}^{q}\nu_{E}(x)_{i}
\left(\int_{S_{\mathbb{G}}^{+}(\nu_{E}(x))}X_{i}h(1,z^{-1})dz\right)  d{\mathcal{S}}_{\infty}^{Q-1}(x)
\\
&  =2\theta_{\infty}\int_{\partial_{\mathbb{G}}^{\ast}E}\sum_{i=1}^{q}
\nu_{E}(x)_{i}\left(  \int_{S_{\mathbb{G}}^{-}(\nu_{E}(x))}X_{i}h(1,z)dz\right)
d{\mathcal{S}}_{\infty}^{Q-1}(x)
\end{align*}
with the obvious meaning of the symbol $S_{\mathbb{G}}^{-}(\nu_{E}(x)).$ Next,
we perform a standard integration by parts in the inner integral, exploiting
the fact that the (Euclidean) outer normal to the halfspace $S_{\mathbb{G}}^{-}(\nu_{E}(x))$ is
\[
n_{E}(x)=\left(  \nu_{E}(x)_{1},...,\nu_{E}(x)_{q},0,...,0\right)
\]
while
\[
X_{i}h(1,z)=\partial_{z_{i}}h\left(  1,z\right)  +\sum_{j=q+1}^{n}
\partial_{z_{j}}\left(  q_{j}^{i}\left(  z\right)  h\left(  1,z\right)\right)
\]
hence 
\[
\sum_{i=1}^{q}
\nu_{E}(x)_{i}\left(\int_{S_{\mathbb{G}}^{-}(\nu_{E}(x))}X_{i}h(1,z)dz\right)  
= \sum_{i=1}^{q}\left(  \nu_{E}(x)_{i}\right)  ^{2}\int_{T_{\mathbb{G}}
(\nu_{E}(x))}h(1,z)dz=\phi_{\mathbb{G}}(\nu_{E})
\]
for any $x\in\partial_{\mathbb{G}}^{\ast}E,$ and we conclude that 
\begin{equation}   \label{mainequality2}
\lim_{t\rightarrow0}\frac{1}{2\theta_{\infty}\sqrt{t}}\int_{E^{c}}W_{t}\chi_{E}(x)dx=
\int_{\partial_{\mathbb{G}}^{\ast}E}\phi_{\mathbb{G}}(\nu_{E})d{\mathcal{S}}_{\infty}^{Q-1}.
\end{equation}
In order to prove the converse, assume that $E$ has finite measure and notice that for any $z$ 
\[
\int_{{\mathbb{R}}^{n}}\chi_{E}(x\circ D(\sqrt{t})z)(1-\chi_{E}(x))dx=
\frac{1}{2}\int_{\mathbb{R}^{n}}\left\vert \chi_{E}(x\circ D(\sqrt{t})z)-\chi_{E}(x)\right\vert dx,
\]
and that the Lebesgue measure is right invariant as well. Therefore
\begin{align*}
\int_{E^{c}} W_{t}\chi_{E}(x)\ dx&=
t^{-Q/2}\int_{{\mathbb{R}}^{n}}\int_{E}h\left(  1,D(1/\sqrt{t})(y^{-1}\circ x)\right)  
\left(  1-\chi_{E}(x)\right)  \ dydx
\\
& =\int_{{\mathbb{R}}^{n}}\int_{{\mathbb{R}}^{n}}
h(1,z)\chi_{E}(x\circ D(\sqrt{t})z^{-1})\left(  1-\chi_{E}(x)\right)  \ dzdx
\\
& =\frac{1}{2}\int_{{\mathbb{R}}^{n}}h(1,w^{-1})\int_{{\mathbb{R}}^{n}}
|\chi_{E}(x\circ D(\sqrt{t})w)-\chi_{E}(x)|\ dxdw
\\
& =\frac{1}{2}\int_{{\mathbb{R}}^{n}}h(1,w)\int_{{\mathbb{R}}^{n}}
|\chi_{E}(x\circ D(\sqrt{t})w)-\chi_{E}(x)|\ dxdw.
\end{align*}
Then, the finiteness of the liminf of the integral in the left hand side of 
(\ref{mainequality2}) implies that there is a sequence $t_k\downarrow 0$ such that
\[
\frac{1}{2}\int_{{\mathbb{R}}^{n}}h(1,w)\int_{{\mathbb{R}}^{n}}
|\chi_{E}(x\circ D(\sqrt{t_k})w)-\chi_{E}(x)|\ dxdw\leq C\sqrt{t_k}
\]
for all $k$. On the other hand, by the lower Gaussian estimates on $h\left(1,w\right)$ 
we can also write
\begin{align*}
&  \frac{1}{2}\int_{{\mathbb{R}}^{n}}h(1,w)
\int_{{\mathbb{R}}^{n}}|\chi_{E}(x\circ D(\sqrt{t_k})w)-\chi_{E}(x)|\ dxdw
\\
&  \geq\frac{1}{2}\int_{\left\vert w\right\vert \leq2}h(1,w)
\int_{{\mathbb{R}}^{n}}|\chi_{E}(x\circ D(\sqrt{t_k})w)-\chi_{E}(x)|\ dxdw
\\
&  \geq c\int_{\left\vert w\right\vert \leq2}
\int_{{\mathbb{R}}^{n}}|\chi_{E}(x\circ D(\sqrt{t_k})w)-\chi_{E}(x)|\ dxdw
\end{align*}
hence
\[
\int_{\left\vert w\right\vert \leq2}
\int_{{\mathbb{R}}^{n}}|\chi_{E}(x\circ D(\sqrt{t_k})w)-\chi_{E}(x)|\ dxdw\leq C\sqrt{t_k}.
\]
So, if we define the function 
\[
\Phi_{t_k}(w)=
\int_{{\mathbb{R}}^{n}}\frac{|\chi_{E}(x\circ D(\sqrt{t_k})w)-\chi_{E}(x)|}{\sqrt{t_k}}\ dx,
\]
by Fatou's Lemma we deduce that
\[
\Phi_0(w)=\liminf_{k\to \infty} \Phi_{t_k}(w)
\]
is integrable on $B=\{|w|\leq 2\}$. So, for any $\varepsilon>0$, there are a set $B_\varepsilon$
of measure less than $\varepsilon$ and a constant $C_\varepsilon>0$ such that for every
$w\in B\setminus B_\varepsilon$, $\Phi_0(w)\leq C_\varepsilon$. In particular, we 
may choose $n$ linearly independent directions $w_1,\ldots,w_n$ with $\Phi_0(w_i)\leq C_\varepsilon$,
and there is $k_0\in{\mathbb N}$ such that
\[
\int_{{\mathbb{R}}^{n}}
|\chi_{E}(x\circ D(\sqrt{t_k})w_i)-\chi_{E}(x)|\ dx\leq (C_\varepsilon+1)\sqrt{t_k},
\qquad \forall\ i=1,\ldots,n,\ k>k_0.
\]
Then by Lemma~\ref{lemmaCap} we can conclude that $E$ has finite perimeter.
\end{proof}

We note that the above proof of Theorem~\ref{diffusionperimeters} simplifies
that of the analogous result in the Euclidean setting given in
\cite{MirPalParPre07Sho}.

With an argument similar to that used in the proof of Proposition 8 in
\cite{Pre04}, it is also possible to prove the following

\begin{proposition}
If $E$ has finite perimeter, then
\begin{equation}    \label{marcestimate}
\frac{1}{4\sqrt{t}}\int_{\mathbb{R}^{n}}\left\vert W_{t}\chi_{E}-\chi_{E}\right\vert dx
\leq c_{\mathbb{G}}P_{\mathbb{G}}\left(  E\right)  ,\text{\ \ }t>0 
\end{equation}
with
\[
c_{\mathbb{G}}=\int_{\mathbb{R}^{n}}\left\vert \nabla_{X}h\left(  1,z\right)
\right\vert dz.
\]
\end{proposition}
\begin{proof}
First of all we note that
\[
\int_{\mathbb{R}^{n}}\left( W_{t}\chi _{E}\left( x\right) -\chi _{E}\left(
x\right) \right) dx=0.
\]
This follows by (iv) in Theorem 1 if $E$ has finite measure; otherwise $E^{c}
$ must have finite measure, and the same identity holds in view of (\ref{decomp}).
Recalling that $0\leq W_{t}\chi_{E}(x)\leq1,$ we deduce that
\begin{align}   \label{42}
\int_{E^{c}}W_{t}\chi_{E}\left(  x\right)  dx & =
\int_{E^{c}}\left(  W_{t}\chi_{E}\left(  x\right)  -\chi_{E}\left(  x\right)  \right)  dx  
=\int_{\mathbb{R}^{n}}\left(  W_{t}\chi_{E}\left(  x\right)  
-\chi_{E}\left(x\right)  \right)  ^{+}dx
\\   \nonumber 
& =\int_{\mathbb{R}^{n}}\left(  W_{t}\chi_{E}\left(  x\right)-\chi_{E}\left(x\right)\right)^{-}dx  
=\frac{1}{2}\int_{\mathbb{R}^{n}}\left\vert W_{t}\chi_{E}\left(  x\right)  
-\chi_{E}\left(  x\right)  \right\vert dx
\end{align}
(notice that the above equality holds under the hypothesis that either $E$ or $E^c$ has 
finite measure). We can also write
\[
\int_{E^{c}}W_{t}\chi_{E}\left(  x\right)  dx=\int_{E^{c}}\left(  W_{t}%
\chi_{E}\left(  x\right)  -\chi_{E}\left(  x\right)  \right)  dx=\int_{0}%
^{t}ds\int_{E^{c}}LW_{t}\chi_{E}\left(  x\right)  dx.
\]
\textrm{ }But then by the divergence theorem%
\begin{align}
\int_{E^{c}}LW_{t}\chi_{E}\left(  x\right)  dx  &  
=\int_{\partial_{\mathbb{G}}^{\ast}E}dP_{\mathbb{G}}\left(  E\right)  \left(  x\right)
\int_{\mathbb{R}^{n}}\langle\nabla_{X}h(s,y^{-1}\circ x),\nu_{E}(x)\rangle
\chi_{E}\left(  y\right)  dy
\nonumber\\     \label{36}
&  \leq\int_{\partial_{\mathbb{G}}^{\ast}E}dP_{\mathbb{G}}\left(  E\right)
\left(  x\right)  \int_{\mathbb{R}^{n}}\left\vert \nabla_{X}
h(s,y^{-1}\circ x)\right\vert dy
\end{align}
By (\ref{42}) and (\ref{36}) the conclusion then follows since 
\begin{align*}
\frac{1}{4\sqrt{t}}\int_{{\mathbb{R}}^{n}}|W_{t}\chi_{E}- \chi_{E}|dx
&=\frac{1}{2\sqrt{t}}\int_{E^{c}}W_{t}\chi_{E}(x)dx=\frac{1}{2\sqrt{t}}
\int_{0}^{t}ds\int_{E^{c}}LW_{s}\chi_{E}(x)dx
\\
&  \leq\frac{1}{2\sqrt{t}}\int_{0}^{t}\left(  \int_{\partial_{{\mathbb{G}}}^{\ast}E}
dP_{\mathbb{G}}\left(  E\right)  (x)\int_{{\mathbb{R}}^{n}} 
|\nabla_{X}h(s,y^{-1}\circ x)|dy\right)  ds
\\
&  =\frac{1}{2\sqrt{t}}\int_{0}^{t}\left(  \int_{\partial_{{\mathbb{G}}}^{\ast}E}
dP_{\mathbb{G}}\left(  E\right)  (x)
\frac{1}{\sqrt{s}}\int_{{\mathbb{R}}^{n}}|\nabla_{X}h(1,z)|dz\right)  ds
\\
&  \leq c_{\mathbb{G}}P_{{\mathbb{G}}}(E).
\end{align*}
\end{proof}

\begin{proof}  [Proof of Theorem~\ref{Main Theorem 2}]The derivation of Theorem~\ref{Main Theorem 2} 
from Theorem~\ref{diffusionperimeters} is based on the
coarea formula, and is similar to that of Theorem 4.1 of
\cite{MirPalParPre07Sho}. Assume $f\in BV({\mathbb{G}})$ and set 
$E_{\tau}=\{f>\tau\}$. By Proposition~\ref{coarea}, for a.e. $\tau\in\mathbb{R}$ the
set $E_{\tau}$ has finite perimeter, hence by Theorem~\ref{diffusionperimeters} we can write 
\[
\lim_{t\rightarrow0}\frac{1}{2\theta_{\infty}\sqrt{t}} 
\int_{E_{\tau}^{c}}W_{t}\chi_{E_{\tau}}(x)dx =
\int_{\partial_{{\mathbb{G}}}^{\ast}E_{\tau}}
\phi_{{\mathbb{G}}}(\nu_{E_{\tau}}(x)) d{\mathcal{S}}_{\infty}^{Q-1}(x)
\]
for a.e. $\tau$.

Moreover, using the same arguments as \cite[Proposition 3.69]{AmbFusPal00Fun}
and the continuity of $\phi_{\mathbb{G}}$ it is easily checked that the
function $\phi_{\mathbb{G}}(\sigma_{f}(x))$ is (obviously bounded and) Borel.
Comparing $f$ with $f\vee\tau\chi_{E_{\tau}}$ and using \cite[Proposition 3.73]{AmbFusPal00Fun} 
we see that for a.e. $\tau\in{\mathbb{R}}$ the equality
$\sigma_{f}(x)=\nu_{E_{\tau}}(x)$ holds for ${\mathcal{S}}_{\infty}^{Q-1}$-a.e. 
$x\in\partial_{\mathbb{G}}^{\ast}E_{\tau}$, whence
\[
\phi_{\mathbb{G}}(\sigma_{f}(x))=\phi_{\mathbb{G}}(\nu_{E_{\tau}}(x))
\text{ for a.e.}\tau\in{\mathbb{R}},\text{ }
{\mathcal{S}}_{\infty}^{Q-1}\text{-a.e.}x\in\partial_{\mathbb{G}}^{\ast}E_{\tau}.
\]
We point out that if $x$ belongs to the boundary of two different level sets
$E_{\tau},E_{\sigma}$ (which happens whenever $f$ has a jump at $x$), then the
above equality holds for both the levels because $\nu_{E_{\tau}}(x)=\nu_{E_{\sigma}}(x)$ 
for $\sigma,\tau\in\lbrack f^{-}(x),f^{+}(x)]$. We
also notice that, for almost every $x,y\in\mathbb{R}^{n}$,
\begin{equation}   \label{miservesubito}
\int_{\mathbb{R}}|\chi_{E_{\tau}}(x)-\chi_{E_{\tau}}(y)|d\tau
=|f(x)-f(y)|.
\end{equation}
With the aid of coarea formula \eqref{coareag} with 
$g(x)=\phi_{\mathbb{G}}(\sigma_{f}(x))$, by (\ref{42}), (\ref{miservesubito}) and using the
dominated convergence theorem, we obtain that 
\begin{align}
\int_{{\mathbb{R}}^{n}}\phi_{\mathbb{G}}(\sigma_{f}(x))d\left\vert D_{\mathbb{G}}f\right\vert 
& = \theta_\infty\int_{\mathbb{R}}\int_{\partial_{\mathbb{G}}^{\ast}E_{\tau}}
\phi_{\mathbb{G}}(\nu_{E_{\tau}}(x))d{\mathcal S}_\infty^{Q-1}(x)d\tau
\nonumber\\
& =\int_{\mathbb{R}}\lim_{t\rightarrow0}\frac{1}{2\sqrt{t}}\int_{E_{\tau}^{c}}W_{t}
\chi_{E_{\tau}}(x)dxd\tau
\nonumber\\
& =\lim_{t\rightarrow0}\int_{\mathbb{R}}\frac{1}{4\sqrt{t}}
\int_{\mathbb{R}^{n}}|W_{t}\chi_{E_{\tau}}-\chi_{E_{\tau}}|dxd\tau
\nonumber\\
& =\lim_{t\rightarrow0}\frac{1}{4\sqrt{t}}\int_{\mathbb{R}}
\int_{\mathbb{R}^{n}\times\mathbb{R}^{n}}|\chi_{E\tau}(y)-\chi_{E_{\tau}}(x)
|h(t,y^{-1}\circ x)dxdyd\tau
\nonumber\\   \label{ultimo}
& =\lim_{t\rightarrow0}\frac{1}{4\sqrt{t}}
\int_{\mathbb{R}^{n}\times\mathbb{R}^{n}}|f(x)-f(y)|h(t,y^{-1}\circ x)dxdy.
\end{align}
Conversely, assume that $f\in L^{1}({\mathbb{R}}^{n})$ and that
\[
\liminf_{t\rightarrow0}\frac{1}{\sqrt{t}}
\int_{\mathbb{R}^{n}\times\mathbb{R}^{n}}|f(x)-f(y)|h(t,y^{-1}\circ x)dxdy
\]
is finite. Then,
\[
\int_{{\mathbb{R}}^{n}\times{\mathbb{R}}^{n}}|f(x)-f(y)|\,h(t,y^{-1}\circ x)dxdy
=\int_{\mathbb{R}}\int_{{\mathbb{R}}^{n}\times{\mathbb{R}}^{n}}
|\chi_{E_{\tau}}(x)-\chi_{E_{\tau}}(y)|\,h(t,y^{-1}\circ x)dxdyd\tau\,.
\]
Since $\int_{E^{c}}W_{t}\chi_{E}\geq0$, by the above equality and (\ref{42}), which we are
allowed to exploit because for $f\in L^1({\mathbb R}^n)$ either $E_\tau$ or its 
complement has finite measure, we get
\begin{align*}
0 &  \leq\int_{\mathbb{R}}\left(  \liminf_{t\rightarrow0}\frac{1}{\sqrt{t}}
\int_{E_{\tau}^{c}}W_{t}\chi_{E_{\tau}}dx\right)  d\tau
\leq\liminf_{t\rightarrow0}\int_{\mathbb{R}}\frac{1}{\sqrt{t}}
\int_{E_{\tau}^{c}}W_{t}\chi_{E_{\tau}}dx\,d\tau
\\
&  \leq\liminf_{t\rightarrow0}\frac{1}{2\sqrt{t}}
\int_{{\mathbb{R}}^{n}\times{\mathbb{R}}^{n}}
\int_{\mathbb{R}}|\chi_{E_{\tau}}(x)-\chi_{E_{\tau}}(y)|\,h(t,y^{-1}\circ x)dx\,dy\,d\tau
<+\infty\,.
\end{align*}
In particular, by Theorem~\ref{diffusionperimeters}, for a.e. $\tau\in{\mathbb{R}}$ 
the set ${E_{\tau}}$ has finite perimeter and the limit
${\displaystyle\lim_{t\rightarrow0}}\frac{1}{\sqrt{t}}
\int_{E_{\tau}^{c}}W_{t}\chi_{E_{\tau}}dx$ exists. As a consequence, 
$f\in BV_{\mathrm{loc}}(\mathbb{G})$ and we may use the above identity (\ref{ultimo}) to get 
\begin{align*}
|D_{\mathbb{G}}f|({\mathbb{R}}^{n}) &  \leq
\frac{1}{\min\phi_{\mathbb{G}}}\int_{\mathbb{R}^{n}}\phi_{\mathbb{G}}(\sigma_{f})
d\left\vert D_{\mathbb{G}}f\right\vert 
\\
&  \leq\frac{1}{\min\phi_{\mathbb{G}}}\liminf_{t\rightarrow0}\frac{1}{4\sqrt{t}}
\int_{{\mathbb{R}}^{n}\times{\mathbb{R}}^{n}}|f(x)-f(y)|\,h(t,y^{-1}\circ x)dx\,dy
<+\infty
\end{align*}
that is, $f\in BV({\mathbb{G}})$.
\end{proof}

\end{document}